\documentclass{amsart}
\usepackage[english]{babel}
\usepackage{amsmath,amsthm}
\usepackage[a4paper]{geometry}
\theoremstyle{plain}
\newtheorem{theorem}{Theorem}[section]

\newtheorem{lemma}[theorem]{Lemma}

\numberwithin{theorem}{section}
\numberwithin{equation}{section}
\numberwithin{table}{section}
\numberwithin{figure}{section}

\theoremstyle{definition}

\usepackage{fancyhdr}
\usepackage{enumerate}
\fancypagestyle{plain}{%
    \fancyhead{} 
}

\date{}
\begin{document}
\title{Optimal bounds for the Neuman-S\'{a}ndor means in terms of geometric and contra-harmonic means}
\author{Tiehong Zhao}
\address{Department of mathematics, Hangzhou Normal University, Hangzhou 310036, P.R. China}%
\email{tiehongzhao@gmail.com}
\author{Yuming Chu}
\address{Department of mathematics, Huzhou Teachers College, Huzhou 313000, P.R. China}%
\email{chuyuming@hutc.zj.cn}
\author{Baoyu Liu}
\address{School of Science, Hangzhou Dianzi University, Hangzhou 310018, P.R. China}%
\email{627847649@qq.com}
\thanks{This research was supported by the Natural Science Foundation of China under Grants 11071069 and 11171307, and the Innovation Team Foundation of the Department of Education of Zhejiang Province under Grant T200924.}
\date{}
\maketitle

\begin{abstract}
In this article, we prove that the double inequality
$$\alpha G(a,b)+(1-\alpha)C(a,b)<M(a,b)<\beta G(a,b)+(1-\beta)C(a,b)$$
holds true for all $a,b>0$ with $a\neq b$ if and only if $\alpha\geq
5/9$ and $\beta\leq 1-1/[2\log(1+\sqrt{2})]=0.4327\cdots$, where
$G(a,b),C(a,b)$ and $M(a,b)$ are respectively the geometric,
contra-harmonic and Neuman-S¨¢ndor means of $a$ and $b$.
\end{abstract}
\section{Introduction}
For $a,b>0$ with $a\neq b$ the Neuman-S\'{a}ndor mean $M(a,b)$ \cite{NS} is defined by
\begin{equation}\label{eqn:1.1}
M(a,b)=\frac{a-b}{2\sinh^{-1}\left(\frac{a-b}{a+b}\right)},
\end{equation}
where $\sinh^{-1}(x)=\log(x+\sqrt{x^2+1})$ is the inverse hyperbolic sine function.

Recently, the Neuman-S\'{a}ndor mean has been the subject of
intensive research. In particular, many remarkable inequalities for
the Neuman-S\'{a}ndor mean  can be found in the literature
\cite{NS,NS2}.

Let $G(a,b)=\sqrt{ab}$, $L(a,b)=(b-a)/(\log b-\log a)$,
$P(a,b)=(a-b)/(4\arctan\sqrt{a/b}-\pi)$, $A(a,b)=(a+b)/2$,
$T(a,b)=(a-b)/[2\arcsin((a-b)/(a+b))]$, $Q(a,b)=\sqrt{(a^2+b^2)/2}$
and $C(a,b)=(a^2+b^2)/(a+b)$ be the geometric, logarithmic, first
Seiffert, arithmetic, second Seiffert, quadratic and contra-harmonic
means of $a$ and $b$, respectively. Then it is well-known that the
inequalities
$$G(a,b)<L(a,b)<P(a,b)<A(a,b)<M(a,b)<T(a,b)<Q(a,b)<C(a,b)$$
hold true for $a,b>0$ with $a\neq b$.

Neuman and S\'{a}ndor \cite{NS,NS2} established that
$$A(a,b)<M(a,b)<T(a,b),$$
$$P(a,b)M(a,b)<A^2(a,b),$$
$$A(a,b)T(a,b)<M^2(a,b)<[A^2(a,b)+T^2(a,b)]/2$$
hold true for all $a,b>0$ with $a\neq b$.

Let $0<a,b<1/2$ with $a\neq b$, $a'=1-a$ and $b'=1-b$. Then the following Ky Fan inequalities
$$\frac{G(a,b)}{G(a',b')}<\frac{L(a,b)}{L(a',b')}<\frac{P(a,b)}{P(a',b')}<\frac{A(a,b)}{A(a',b')}<\frac{M(a,b)}{M(a',b')}<\frac{T(a,b)}{T(a',b')}$$
were presented in \cite{NS}.

Let $L_p(a,b)=\left[(b^{p+1}-a^{p+1})/((p+1)(b-a))\right]^{1/p}
(p\neq-1,0)$, $L_0=1/e(b^b/a^a)^{1/(b-a)}$ and
$L_{-1}(a,b)=(b-a)/(\log b-\log a)$ be the th generalized
logarithmic mean of $a$ and $b$. Li et al. \cite{LLC} showed that
the double inequality $L_{p_0}(a,b)<M(a,b)<L_2(a,b)$ holds true for
all $a,b>0$ with $a\neq b$, where $p_0=1.843\cdots$ is the unique
solution of the equation $(p+1)^{1/p}=2\log(1+\sqrt{2})$.

In \cite{Ne}, Neuman proved that the double inequalities
$$\alpha Q(a,b)+(1-\alpha)A(a,b)<M(a,b)<\beta Q(a,b)+(1-\beta)A(a,b)$$
and
$$\lambda Q(a,b)+(1-\lambda)A(a,b)<M(a,b)<\mu Q(a,b)+(1-\mu)A(a,b)$$
hold for all $a,b>0$ with $a\neq b$ if and only if $\alpha\leq(1-\log(\sqrt{2}+1))/[(\sqrt{2}-1)\log(\sqrt{2}+1)]=0.3249\cdots$, $\beta\geq1/3$,
$\lambda\leq(1-\log(\sqrt{2}+1))/\log(\sqrt{2}+1)=0.1345\cdots$ and $\mu\geq1/6$.

The main purpose of this paper is to find the least value $\alpha$ and the greatest value $\beta$ such that the double inequality
$$\alpha G(a,b)+(1-\alpha)C(a,b)<M(a,b)<\beta G(a,b)+(1-\beta)C(a,b)$$
holds for all $a,b>0$ with $a\neq b$.

Our main result is presented in Theorem \ref{th:1.1}.

\begin{theorem}\label{th:1.1}The inequality
\begin{equation}\label{eqn:1.2}
\alpha G(a,b)+(1-\alpha)C(a,b)<M(a,b)<\beta G(a,b)+(1-\beta)C(a,b)
\end{equation}
holds true for all $a,b>0$ with $a\neq b$ if and only if $\alpha\geq
5/9$ and $\beta\leq 1-1/[2\log(1+\sqrt{2})]=0.4327\cdots$.
\end{theorem}

\section{Lemmas}

In order to prove our main result we need a lemma, which we present in this section.

\medskip

\begin{lemma}\label{le:2.1}
Let $p\in(0,1)$, $\lambda_0=1-1/[2\log(1+\sqrt{2})]=0.4327\cdots$ and
\begin{equation}\label{eqn:2.1}
\varphi_p(t)=\sinh^{-1}(\sqrt{1-t^2})-\frac{\sqrt{1-t^2}}{(p-1)t^2+pt+2-2p}
\end{equation} Then $\varphi_{5/9}(t)<0$ and $\varphi_{\lambda_0}(t)>0$ for all $t\in(0,1)$.
\end{lemma}

\begin{proof}
From (\ref{eqn:2.1}), we have
\begin{align}\label{eqn:2.2}
\varphi_p(1)&=0,\\ \label{eqn:2.3}
\varphi_p(0)&=\log(1+\sqrt{2})-\frac{1}{2(1-p)},\\ \label{eqn:2.4}
\varphi'_p(t)&=\frac{f_p(x)}{\sqrt{1-t^2}\sqrt{2-t^2}[(p-1)t^2+pt+2-2p]^2}\quad t\in(0,1)
\end{align}
where
\begin{align}\label{eqn:2.5}
\begin{split}
f_p(t)=&-t[(p-1)t^2+pt+2-2p]^2+[2(p-1)t+p](1-t^2)\sqrt{2-t^2}\\
&\hspace{4cm}+t\sqrt{2-t^2}[(p-1)t^2+pt+2-2p].
\end{split}
\end{align}

\medskip
We divide the proof into two cases.
\setlength\leftmargini{.5cm}
\begin{description}
\item[Case 1] $p=5/9$. Then (\ref{eqn:2.5}) leads to
\begin{align}\label{eqn:2.6}
f_{5/9}(1)&=0,\\ \label{eqn:2.7}
f'_{5/9}(t)&=\frac{g_{5/9}(t)}{81\sqrt{2-t^2}}
\end{align}
where
\begin{equation}\label{eqn:2.8}
g_{5/9}(t)=-45t+216t^2-144t^4-(64+160t-117t^2-160t^3+80 t^4)\sqrt{2-t^2}
\end{equation}

\medskip

We divide the discussion of this case into two subcases.
\begin{description}
\item[Subcase 1.1] $t\in(0,3/4]$. Then from (\ref{eqn:2.8}) we clearly see that
\begin{equation}\label{eqn:2.9}
\frac{g_{5/9}(t)}{\sqrt{2-t^2}}\leq216t^2-\frac{45t+144t^4}{\sqrt{2}}-(64+160t-117t^2-160t^3+80 t^4):=\mu(t)
\end{equation}
Differentiating $\mu(t)$ yields
\begin{align}\label{eqn:2.10}
\mu'(t)&=-160-\frac{45}{\sqrt{2}}+666t+480t^2-\left(288\sqrt{2}+320\right)t^3,\\
\label{eqn:2.11} \mu''(t)&=666+960t-(960+864\sqrt{2})t^2.
\end{align}
From (\ref{eqn:2.10}) and scientific computation we know that there exists unique $t_1=0.25869\cdots$ in $(0,3/4]$ satisfying the equation $\mu'(t_1)=0$. It follows from (\ref{eqn:2.11}) that $\mu''(t_1)=768.33\cdots>0$ and $t_1$ is a unique extreme minimum point of $\mu(t)$ in $(0,3/4]$. Therefore we obtain
\begin{equation}\label{eqn:2.12}
\mu(t)\leq\max\{\mu(0),\mu(3/4)\}=-10.5824\cdots<0.
\end{equation}
Inequalities (\ref{eqn:2.9}) and (\ref{eqn:2.12}) lead to the conclusion that
$$g_{5/9}(t)<0$$ for $t\in(0,3/4]$.

\medskip

\item[Subcase 1.2] $t\in(3/4,1)$. Then (\ref{eqn:2.8}) gives
\begin{align}\label{eqn:2.13}
\begin{split}
g'_{5/9}(t)\sqrt{2-t^2}&=-320+532t+1280t^2-991t^3-640t^4+400t^5\\
&\hspace{2.7cm}-(45-432t+576t^3)\sqrt{2 - t^2}:= h_{5/9}(t),\\
\end{split}
\end{align}
From (\ref{eqn:2.13}) we get
\begin{align}\label{eqn:2.14}
h_{5/9}(1)&=72,\\ \label{eqn:2.15}
h'_{5/9}(t)&=\frac{A(t)}{\sqrt{2-t^2}}+B(t)
\end{align}
where \begin{align}\label{eqn:2.16}
A(t)&=864 + 45 t - 4320 t^2 + 2304 t^4,\\ \label{eqn:2.17}
B(t)&=532 + 2560 t - 2973 t^2 - 2560 t^3 + 2000 t^4.
\end{align}
Equation (\ref{eqn:2.16}) leads to
\begin{align}\label{eqn:2.18}
A'(3/4)&=-2547<0,\\ \label{eqn:2.19} A'(1)&=621>0,\\
\label{eqn:2.20} A''(t)&=-8640 + 27648 t^2>0
\end{align}
for $3/4<t<1$.

\medskip

From (\ref{eqn:2.18})-(\ref{eqn:2.20}) we clearly see that there
exists $t_2\in(3/4,1)$ such that $A(t)$ is strictly decreasing in
$(3/4,t_2]$ and strictly increasing in $[t_2,1)$. So we get
\begin{equation}\label{eqn:2.21}
A(t)\leq\max\{A(3/4,A(1)\}=-803.25\cdots<0.
\end{equation}

\medskip

Equation (\ref{eqn:2.15}) and inequality (\ref{eqn:2.21}) lead to
\begin{equation}\label{eqn:2.22}
h'_{5/9}(t)<\frac{A(t)}{\sqrt{2}}+B(t):=\eta(t).
\end{equation}

\medskip

Computing $\mu(t)$ yields
\begin{align}\label{eqn:2.23}
\eta(3/4)&=-235.484\cdots<0,\\ \label{eqn:2.24}
\eta'(1)&=-2626.886\cdots<0,\\ \label{eqn:2.25}
\eta''(3/4)&=921.522\cdots>0,\\ \label{eqn:2.26}
\eta'''(t)=384[-40+(125+72\sqrt{2})t]&>\eta'''(3/4)=49965.132\cdots>0
\end{align}
for $t\in(3/4,1)$.

\medskip

Inequalities (\ref{eqn:2.23})-(\ref{eqn:2.26}) implies that $\eta(t)<0$ for $t\in(3/4,1)$. From (\ref{eqn:2.8}), (\ref{eqn:2.13}), (\ref{eqn:2.14}) and (\ref{eqn:2.22}), we see
that $g_{5/9}(t)$ is strictly increasing in $(3/4,1)$. So we obtain
$$g_{5/9}(t)<g_{5/9}(1)=0$$ for $3/4<t<1$.
\end{description}

\medskip

Combining the last conclusions in subcases 1.1 and 1.2 we have
\begin{equation}\label{eqn:2.27}
g_{5/9}(t)<0
\end{equation}
for all $t\in(0,1)$.

\medskip

Therefore, $\varphi_{5/9}(t)<0$ for all $t\in(0,1)$ follows easily from (\ref{eqn:2.2}), (\ref{eqn:2.4}), (\ref{eqn:2.6}), (\ref{eqn:2.7}) and (\ref{eqn:2.27}).

\medskip

\item[Case 2] $p=\lambda_0=1-1/[2\log(1+\sqrt{2})]=0.4327\cdots$. Then (\ref{eqn:2.3}) becomes
\begin{equation}\label{eqn:2.28}
\varphi_{\lambda_0}(0)=0
\end{equation}
and (\ref{eqn:2.5}) leads to
\begin{align}\label{eqn:2.29}
f_{\lambda_0}(0)&=\sqrt{2}\lambda_0>0,\quad f_{\lambda_0}(1)=0\\ \label{eqn:2.30}
f'_{\lambda_0}(t)&=\frac{1}{4\delta^2}\left[\frac{C(t)}{\sqrt{2-t^2}}+D(t)\right]
\end{align}
where $$\delta=\frac{1}{2(1-\lambda_0)}=\log(1+\sqrt{2})=0.88137\cdots,$$
\begin{align}\label{eqn:2.31}
C(t)&=2\delta t(1-2\delta+6t-4t^3),\\ \label{eqn:2.32}
D(t)&=-4+(8-16\delta)t+(9+12\delta-12\delta^2)t^2+(16\delta-8)t^3-5t^4.
\end{align}

\medskip

We divide the discussion of this case into two subcases.

\begin{description}
\item[Subcase 2.1]$t\in(0,1/2]$. Then differentiating (\ref{eqn:2.32}) yields
\begin{align}\label{eqn:2.33}
D'(t)&=8-16\delta+(18+24\delta-24\delta^2)t+(48\delta-24)t^2-20t^3,\\ \label{eqn:2.34}
D''(t)&=6[3+4\delta-4\delta^2+(16\delta-8)t-10t^2].
\end{align}
From (\ref{eqn:2.33}) and scientific computation we know that there exists unique $t_3=0.2555\cdots$ in $(0,1/2]$ satisfying the equation
$D'(t_3)=0$. It follows from $(\ref{eqn:2.34})$ that $D''(t_3)=25.9469\cdots>0$ and $t_3$ is a unique extreme minimum point of $D(t)$ in $(0,1/2]$. So we have
\begin{equation}\label{eqn:2.35}
D(t)\leq\max\{D(0),D(1/2)\}=-4<0
\end{equation}
for $0<t<1/2$.

\medskip

From (\ref{eqn:2.23}) and (\ref{eqn:2.35}) one has
\begin{equation}\label{eqn:2.36}
4\delta^2\sqrt{2-t^2}f'_{\lambda_0}(t)=C(t)+D(t)\sqrt{2-t^2}<C(t)+D(t):=
E(t).
\end{equation}
It follows from (\ref{eqn:2.31}) and (\ref{eqn:2.32}) together with (\ref{eqn:2.36}) that
\begin{align}\label{eqn:2.37}
E'(t)&=8-14\delta-4\delta^2+(18+48\delta-24\delta^2)t+(48\delta-24)t^2-(20+32\delta)t^3,\\ \label{eqn:2.38}
E''(t)&=6[3+8\delta-4\delta^2)t+(16\delta-8)t-(10+16\delta)t^2].
\end{align}
Computational and numerical experiments together with (\ref{eqn:2.37}) show that there exists unique $t_4=0.17164\cdots$ in $(0,1/2]$ satisfying the equation $E'(t_4)=0$. It follows from (\ref{eqn:2.38}) that $E''(t_4)=43.686\cdots>0$ and $t_4$ is a unique extreme minimum point of $E(t)$ in $(0,1/2]$. Thus we obtain
\begin{equation}\label{eqn:2.39}
E(t)\leq\max\{E(0),E(1/2)\}=-2.50591\cdots<0
\end{equation}
for $t\in(0,1/2]$.

\medskip

Inequalities (\ref{eqn:2.36}) and (\ref{eqn:2.39}) give
\begin{equation*}
f'_{\lambda_0}(t)<0
\end{equation*}
for $t\in(0,1/2]$.

\medskip

\item[Subcase 2.2]$t\in(1/2,1)$. Then (\ref{eqn:2.31}) leads to
\begin{align}\label{eqn:2.40}
C'(t)&=2\delta(1-2\delta+12t-16t^3),\\ \label{eqn:2.41}
C''(t)&=24\delta(1-4t^2)<0.
\end{align}
It follows from (\ref{eqn:2.40}) and scientific computation that there exists unique $t_5=0.8322\cdots$ in $(1/2,1)$ satisfying the equation $C'(t_5)=0$. Then (\ref{eqn:2.41}) leads to the conclusion that $C(t)$ is strictly increasing in $(1/2,t_5]$ and strictly decreasing in $(t_5,1)$. Hence, we have
\begin{equation}\label{eqn:2.42}
C(t)\geq\inf\{C(1/2),C(1)\}=1.531\cdots>0
\end{equation}
for $t\in(1/2,1)$.

\medskip

On the other hand, equation (\ref{eqn:2.34}) gives
\begin{align}\label{eqn:2.43}
D''(1/2)&=23.81\cdots>0,\\ \label{eqn:2.44} D''(1)&=-2.87\cdots<0,\\
\label{eqn:2.45} D'''(t)&=24(4\delta-2-5t)<12(8\delta-9)<0
\end{align}
for $t\in(1/2,1)$.

\medskip

Equations (\ref{eqn:2.43})-(\ref{eqn:2.45}) imply that there exists
$t_6\in(1/2,1)$ such that $D'(t)$ is strictly increasing in
$(1/2,t_6]$ and strictly decreasing in $[t_6,1)$. Therefore, we have
$D'(t)\geq\inf\{D'(1/2),D'(1)\}=6.229\cdots>0$ for $t\in(1/2,1)$,
which implies that $D(t)$ is strictly increasing in $(1/2,1)$.

\medskip

From the monotonicity of $D(t)$ in $(1/2,1)$ and $C(t)$ in
$(1/2,t_5]$ together with (\ref{eqn:2.42}) we know that
$C(t)/\sqrt{2-t^2}+D(t)$ is strictly increasing in $(1/2,t_5]$.
Equations (\ref{eqn:2.30})-(\ref{eqn:2.32}) lead to
\begin{equation}\label{eqn:2.46}
f'_{\lambda_0}(1/2)=-0.926\cdots<0,\quad
f'_{\lambda_0}(t_4)=0.5193\cdots>0.
\end{equation}
It follows from (\ref{eqn:2.30}) and (\ref{eqn:2.46}) together with
the monotonicity of $C(t)$ in $(1/2,t_5]$ that there exists
$t_7\in(1/2,t_5)$ such that $f'_{\lambda_0}(t)<0$ for
$t\in(1/2,t_7)$ and $f'_{\lambda_0}(t)>0$ for $t\in(t_7,t_5]$.

\medskip

If $t\in(t_5,1)$, then from the monotonicity of $C(t)$ in $[t_5,1)$ and $D(t)$ in $(1/2,1)$ we have
\begin{equation}\label{eqn:2.47}
\frac{C(t)}{\sqrt{2-t^2}}+D(t)>\frac{C(1)}{\sqrt{2}}+D(t_4)=0.6859\cdots>0
\end{equation}

Equation (\ref{eqn:2.30}) and (\ref{eqn:2.47}) lead to $f'_{\lambda_0}(t)>0$ for $t\in(t_5,1)$. Therefore, we know that $f'_{\lambda_0}(t)<0$ for $t\in(1/2,t_7)$ and $f'_{\lambda_0}(t)>0$ for $t\in(t_7,1)$.

\medskip

Combining the last conclusions in subcases 2.1 and 2.2 we clearly
see that $f_{\lambda_0}(t)$ is strictly decreasing in $(0,t_7]$ and
strictly increasing in $[t_7,1)$. Then (\ref{eqn:2.4}) and
(\ref{eqn:2.29}) lead to the conclusion that there exists
$t_0\in(0,t_7)$ such that $\varphi_{\lambda_0}(t)$ is strictly
increasing in $(0,t_0]$ and strictly decreasing in $[t_0,1)$.
\end{description}

\medskip

Therefore, $\varphi_{\lambda_0}(t)>0$ for all $t\in(0,1)$ follows from (\ref{eqn:2.2}) and (\ref{eqn:2.28}) together with the monotonicity of $\varphi_{\lambda_0}(t)$.
\end{description}

\end{proof}

\section{Proof of Theorem 1.1}

\bigskip

\begin{proof}[\bf Proof of Theorem 1.1]
Since $M(a,b),G(a,b)$ and $C(a,b)$ are symmetric and homogeneous of
degree 1. Without loss of generality, we assume that $a>b$. Let
$x=(a-b)/(a+b)\in(0,1)$, $0<p<1$ and
$\lambda_0=1-1/[2\log(1+\sqrt{2})]=0.4327\cdots$. Then
\begin{equation}\label{eqn:3.1}
\frac{C(a,b)-M(a,b)}{C(a,b)-G(a,b)}=\frac{(1+x^2)\sinh^{-1}(x)-x}{(1+x^2-\sqrt{1-x^2})\sinh^{-1}(x)}
\end{equation}
and
\begin{align}\label{eqn:3.2}
\begin{split}
\lefteqn{pG(a,b)+(1-p)C(a,b)-M(a,b)}\\
&=A(a,b)\left[p\sqrt{1-x^2}+(1-p)(1+x^2)-\frac{x}{\sinh^{-1}(x)}\right]\\
&=\frac{A(a,b)[p\sqrt{1-x^2}+(1-p)(1+x^2)]}{\sinh^{-1}(x)}\varphi_p(\sqrt{1-x^2})
\end{split}
\end{align}
where $\varphi_p(t)$ is defined as in Lemma \ref{le:2.1}.

Note that
\begin{align}\label{eqn:3.3}
\lim\limits_{x\rightarrow0^+}\frac{(1+x^2)\sinh^{-1}(x)-x}{(1+x^2-\sqrt{1-x^2})\sinh^{-1}(x)}&=\frac{5}{9},\\ \label{eqn:3.4}
\lim\limits_{x\rightarrow1^-}\frac{(1+x^2)\sinh^{-1}(x)-x}{(1+x^2-\sqrt{1-x^2})\sinh^{-1}(x)}&=1-\frac{1}{2\log(1+\sqrt{2})}=\lambda_0.
\end{align}

Equation (\ref{eqn:3.2}) and Lemma \ref{le:2.1} lead to the conclusion that the double inequality
$$\frac{5}{9}G(a,b)+\frac{4}{9}C(a,b)<M(a,b)<\lambda_0G(a,b)+(1-\lambda_0)C(a,b)$$
holds for all $a,b>0$ with $a\neq b$.
\setlength\leftmargini{.4cm}
\begin{itemize}
\item If $p_1<5/9$, then equations (\ref{eqn:3.1}) and (\ref{eqn:3.3}) imply that there exists $0<\delta_1<1$ such that $M(a,b)<p_1 G(a,b)+(1-p_1)C(a,b)$
for all $a,b>0$ with $(a-b)/(a+b)\in(0,\delta_1)$.

\medskip

\item If $p_2>\lambda_0$, then equations (\ref{eqn:3.1}) and (\ref{eqn:3.4}) imply that there exists $0<\delta_2<1$
such that $M(a,b)>p_2 G(a,b)+(1-p_2)C(a,b)$ for all $a,b>0$ with
$(a-b)/(a+b)\in(1-\delta_2,1)$.
\end{itemize}
Therefore, we conclude that in order for the inequalities
(\ref{eqn:1.2}) to be valid it is necessary and sufficient that
$\alpha\geq 5/9$ and $\beta\leq
1-1/[2\log(1+\sqrt{2})]=0.4327\cdots$.
\end{proof}


\begin{thebibliography}{99}
\addtolength{\itemsep}{-.3em}


\bibitem{NS}
  E. \textsc{Neuman} and J. \textsc{S\'{a}ndor}, {\it On the Schwab-Borchardt mean}, Math. Pannon. \textbf{14}, 2(2003), 253-266.

\bibitem{NS2}
  E. \textsc{Neuman} and J. \textsc{S\'{a}ndor}, {\it On the Schwab-Borchardt mean II}, Math. Pannon \textbf{17}, 1(2006), 49-59.

\bibitem{LLC}
Y.-M. \textsc{Li}, B.-Y. \textsc{Long} and Y.-M. \textsc{Chu}, {\it
Sharp bounds for the Neuman-S\'{a}ndor mean in terms of generalized
logarithmic mean}, J. Math. Inequal. \textbf{6}, 4(2012), 567-577.

\bibitem{Ne}
 E. \textsc{Neuman}, {\it A note on a certain bivariate mean}, J. Math. Inequal. \textbf{6}, 4(2012), 637-643.


 \end{thebibliography}
\end{document}